\documentclass[reqno]{amsart}
\usepackage{amssymb,amsmath,amsthm,amscd,latexsym,amsfonts}
\usepackage{mathtools}
\usepackage[english]{babel}
\usepackage[T1]{fontenc}
\usepackage{inputenc}
\usepackage[arrow,matrix]{xy}
\usepackage{amsaddr}
\usepackage{graphicx}
\usepackage{cite}
\newtheorem{thm}{Theorem}[section]
\newtheorem{cor}[thm]{Corollary}

\newtheorem{prop}[thm]{Proposition}

\theoremstyle{definition}
\newtheorem{defn}[thm]{Definition}

\theoremstyle{remark}

\numberwithin{equation}{section}

\DeclareMathOperator{\Hom}{Hom}
\DeclareMathOperator{\Orb}{Orb}
\DeclareMathOperator{\Ann}{Ann}
\DeclareMathOperator{\Ker}{Ker}
\DeclareMathOperator{\Img}{Im}

\DeclareMathOperator{\cent}{Center}
\DeclareMathOperator{\spann}{span}
\DeclareMathOperator{\GL}{GL}
\DeclareMathOperator{\Leib}{\texttt{Leib}}
\begin{document}

\title[Solvable Leibniz algebras with quasi-filiform Lie algebras nilradicals]{Solvable Leibniz algebras with quasi-filiform Lie
 algebras of maximum length nilradicals}

\author{Kh.~A.~Khalkulova\textsuperscript{1}, M.~Ladra\textsuperscript{2}, B.~A.~Omirov\textsuperscript{3}, A.~M.~Sattorov\textsuperscript{4}}

\address{\textsuperscript{1,4}Institute of Mathematics, 81, Mirzo Ulug'bek street, 100125, Tashkent, Uzbekistan, xalkulova@gmail.com, saloberdi90@mail.ru}
\address{\textsuperscript{2}Department of Mathem\'aticas, Institute of Matem\'aticas, University of Santiago de Compostela, 15782, Spain, manuel.ladra@usc.es}
\address{\textsuperscript{3} National University of Uzbekistan, 4, University street, 100174, Tashkent,  Uzbekistan, omirovb@mail.ru}

\begin{abstract} In this paper solvable Leibniz algebras whose nilradical is quasi-filiform Lie algebra of maximum length, are classified.
The rigidity of such Leibniz algebras with two-dimensional complemented space to nilradical is proved.

\end{abstract}



\maketitle





\section{Introduction}

Leibniz algebras are ``non-commutative'' analogue of Lie algebras. They were introduced about twenty five years ago by J.-L. Loday in his cyclic homology study \cite{Loday,Loday1}. Since that the theory of Leibniz algebras has been actively studied and many results of the theory of Lie algebras have been extended to Leibniz algebras. Leibniz algebras inherit an important property of Lie algebras, the right multiplication operator of a Leibniz algebra is a derivation.

Thanks to Levi-Malcev's theorem, the solvable Lie algebras \cite{Mub,Mal} have played an important role in the theory of Lie algebras over the last years, either in the classification theory or in geometrical and analytical applications. The investigation of solvable Lie algebras with special types of nilradicals was the subject
of various paper \cite{AnCaGa1,AnCaGa1x,AnCaGa2,Boyko,Campoamor,Mub,TrWi,WaLiDe,NdWi,SnWi}. In Leibniz algebras, the analogue of Levi-Malcev's theorem was recently proved in \cite{Bar}, thus solvable Leibniz algebras also play a central role in the study of Leibniz algebras. In particular, the classifications of $n$-dimensional solvable Leibniz algebras with some restriction on their nilradicals have been obtained (see \cite{Adashev,Adashev2,Casas,Abror,Nulfilrad,aloberdi}).

Since the description of finite-dimensional solvable Lie algebras is a boundless problem, lately  geometric approaches have been  developed. Relevant tools of geometric approaches are Zariski topology and natural action of linear reductive group on varieties of algebras in a such way that orbits under the action consists of isomorphic algebras.
It is a well-known result of algebraic geometry that any algebraic variety (evidently, algebras defined via identities form an algebraic variety) is a union of a finite number of irreducible components.
The most important algebras are those whose orbits under the action are open sets in Zariski topology.
The algebras of a variety with open orbits are important since the closures of orbits of such algebras form irreducible components of the variety. At the same time, there are irreducible components which cannot be obtained via closure of the orbit of any algebra. This fact does not detract the importance of algebras with open orbits. This is a motivation of many works focused to discovering  algebras with open orbits and to describe sufficient properties of such algebras \cite{Burde,O'Halloran1,O'Halloran2}.

The purpose of the present work is to continue the study of solvable Leibniz algebras with a given nilradical. Thanks to work \cite{Gomez} the quasi-filiform Lie algebras of maximum length are known and we focus our attention to the description of the solvable Leibniz algebras with quasi-filiform Lie algebras of maximum length nilradicals. In order to achieve our description we use the method which involve non-nilpotent outer derivations of the nilradicals.

Throughout the paper vector spaces and algebras are  finite-dimensional over the field of the complex numbers. Moreover, in the table of multiplication  of an algebra the omitted products are assumed to be zero and, if it is not noted, we   consider non-nilpotent solvable algebras.

\section{Preliminaries}

In this section we give necessary definitions and preliminary results.

\begin{defn}[\cite{Loday}] A  vector space $L$ over a field $\mathbb{F}$ equipped with bilinear bracket $[-,-]$ is called a \emph{Leibniz algebra} if for any $x,y,z\in L$ the so-called Leibniz identity
\[ \big[x,[y,z]\big]=\big[[x,y],z\big] - \big[[x,z],y\big] \]  holds.
\end{defn}

For examples of Leibniz algebras we refer to papers \cite{Adashev2,Abror,Loday,Loday1} and references therein.

Further we   use the following notation
\[ {\mathcal L}(x, y, z)=[x,[y,z]] - [[x,y],z] + [[x,z],y].\]
It is obvious that Leibniz algebras are determined by the identity
${\mathcal L}(x, y, z)=0$.

From the Leibniz identity we conclude that the elements $[x,x], [x,y]+[y,x]$ for any $x, y \in L$ lie in  $\Ann_r(L) =\{x \in L \mid [y,x] = 0,\ \text{for \ all}\ y \in L \}$, the \emph{right annihilator} of the Leibniz algebra $L$. Note that $\Ann_r(L)$ is a two-sided ideal of $L$.

The set $\cent(L)=\{x\in L  \mid   [x,y]=0=[y,x], \ \text{for \ all}\ y \in L \}$ is said to be the \emph{center} of $L$ and it forms two-sided ideal of $L$.

The notion of derivation for the Leibniz algebras case is defined as usual, that is,
 a linear map $d \colon L \rightarrow L$ of a Leibniz algebra $L$ is called a \emph{derivation} if it satisfies the condition
\begin{equation}\label{eq0}
d([x,y])=[d(x),y] + [x, d(y)], \ \forall x, y \in L.
\end{equation}

The Leibniz identity implies that the right multiplication operator $\mathcal{R}_x \colon L \to L, \mathcal{R}_x(y)=[y,x],  y \in L$, is a derivation.

For a given Leibniz algebra $L$ \emph{the lower central}  and \emph{derived series} defined as follows:
\[L^1=L, \ L^{k+1}=[L^k,L],  \ k \geq 1, \qquad \qquad
L^{[1]}=L, \ L^{[s+1]}=[L^{[s]},L^{[s]}], \ s \geq 1,\]
respectively.

\begin{defn}
 A Leibniz algebra $L$ is said to be
\emph{nilpotent} (respectively, \emph{solvable}), if there exists $k\in\mathbb N$ ($s\in\mathbb N$) such that $L^{k}=\{0\}$ (respectively, $L^{[s]}=\{0\}$).
 The minimal number $k$ with such property is said to be the \emph{index of nilpotency} of the algebra $L$.
\end{defn}

Clearly, the index of nilpotency of an $n$-dimensional nilpotent Leibniz algebra is not greater than $n+1$.

Recall, the maximal nilpotent ideal of a Leibniz algebra is said to be the \emph{nilradical} of the  algebra.

Let $R$ be a solvable Leibniz algebra with a nilradical $N$. We denote by $Q$ the complementary vector space of the nilradical $N$ to the algebra $R$. Let us consider the restrictions to $N$ of the right multiplication operator on an element $x \in Q$ (denoted by $\mathcal{R}_{{x |}_{N}}$).
  From \cite{Nulfilrad} we know that for any $x \in Q$, the operator $\mathcal{R}_{{x |}_{N}}$ is a non-nilpotent  derivation of $N$.

Let $\{x_1, \dots, x_m\}$ be a basis of $Q$, then for any scalars $\{\alpha_1, \dots, \alpha_m\}\in
\mathbb{C}\setminus\{0\}$, the matrix $\alpha_1\mathcal{R}_{{x_1 |}_{N}}+\dots+\alpha_m\mathcal{R}_{{x_m|}_{N}}$ is non-nilpotent, which means that the operators $\{\mathcal{R}_{{x_1 |}_{N}}, \dots, \mathcal{R}_{{x_m |}_{N}}\}$ are nil-independent \cite{Mub}. Therefore, we have that the dimension of $Q$ is bounded by the maximal number of nil-independent derivations of the nilradical $N$ (see \cite[Theorem 3.2]{Nulfilrad}).

Below we define the notion of quasi-filiform Leibniz algebra.

\begin{defn}
An $n$-dimensional Leibniz algebra is called \emph{quasi-filiform} if its index of nilpotency is equal
to $n-1$.
\end{defn}



A Leibniz algebra $L$ is $\mathbb{Z}$-graded, if $L =\bigoplus_{i\in \mathbb{Z}}V_i$, where $[V_i, V_j]\subseteq V_{i+j}$ for any $i, j \in \mathbb{Z}$ with a finite number of non-null spaces $V_i$.

We say that a nilpotent Leibniz algebra $L$ admits \emph{the connected gradation} $L = V_{k_1}\oplus\dots\oplus V_{k_t}$, if $V_{k_i}\neq\{0\}$ for any $i \  (1\leq i\leq t)$.

\begin{defn}
The number $l(\oplus L)=l(V_{k_1}\oplus\dots\oplus V_{k_t})=k_t-k_1 + 1$ is called \emph{the
length of gradation}. A gradation is called \emph{of maximum length}, if $l(\oplus L) = \dim(L)$.

\end{defn}

We denote by $l(L) = \max\{l(\oplus L)\ \text{such that}\  L = V_{k_1}\oplus\dots\oplus V_{k_t}\  \text{is a connected gradation}\}$ \emph{the length of an algebra} $L$.

\begin{defn}
 A Leibniz algebra $L$ is called of maximum length if $l(L) = \dim(L)$.
\end{defn}

In the next theorem we present the classification of quasi-filiform Lie algebras of maximum length given in \cite{Gomez}.

\begin{thm} \label{thm26}
 Let $L$ be an $n$-dimensional quasi-filiform Lie algebra
of maximum length. Then the algebra $L$ is isomorphic to one of the following pairwise non-isomorphic algebras:

\[g^1_{(n,1)}:\left\{\begin{array}{ll}
[e_1,e_i]=e_{i+1},& 2\leq i\leq n-2,\\[1mm]
[e_i,e_{n-i}]=(-1)^ie_{n},& 2\ \leq i\leq  \frac{n-1}{2}, \ n\geq5 \ \  and \ \ n \  is \  odd;\\[1mm]
\end{array}\right.\]

\[g^2_{(n,1)}:\left\{\begin{array}{ll}
[e_1,e_i]=e_{i+1},& 2\leq i\leq n-2,\\[1mm]
[e_i,e_n]=e_{i+2},& 2\leq i\leq n-3, \ n\geq5; \\[1mm]
\end{array}\right. \quad g^3_{(n,1)}:\left\{\begin{array}{ll}
[e_1,e_i]=e_{i+1},& 2\leq i\leq {n-2},\\[1mm]
[e_i,e_{n}]=e_{i+2},& 2\leq i\leq {n-3}\\[1mm]
[e_2,e_i]=e_{i+3}, & 3\leq i \leq {n-4}, \ n\geq7;\\[1mm]
\end{array}\right.\]
\[g^1_{7}:\left\{\begin{array}{ll}
[e_1,e_i]=e_{i+1},& 2\leq i\leq 5,\\[1mm]
[e_2,e_i]=e_{i+2},& 3\leq i\leq 4,\\[1mm]
[e_i,e_{7-i}]=(-1)^ie_{7},& 2\leq i\leq 3;\\[1mm]
\end{array}\right. \quad g^2_{9}:\left\{\begin{array}{ll}
[e_1,e_i]=e_{i+1},& 2\leq i\leq 7,\\[1mm]
[e_2,e_i]=e_{i+2},& 3\leq i\leq 4 ,\\[1mm]
[e_2,e_5]=3e_7,\\[1mm]
[e_2,e_6]=5e_8,\\[1mm]
[e_3,e_i]=-2e_{i+3},&4\leq i\leq 5,\\[1mm]
[e_i,e_{9-i}]=(-1)^ie_{9},&2\leq i \leq 4;\\[1mm]
\end{array}\right.\]
\[g^3_{11}:\left\{\begin{array}{ll}
[e_1,e_i]=e_{i+1},& 2\leq i\leq 9,\\[1mm]
[e_2,e_i]=e_{i+2},& 3\leq i\leq 4 ,\\[1mm]
[e_2,e_i]=-e_{i+2},& 6\leq i\leq 7\\[1mm]
[e_3,e_7]=-e_{10},\\[1mm]
[e_3,e_i]=e_{i+3},&4\leq i\leq 5,\\[1mm]
[e_4,e_i]=e_{i+4},&5\leq i\leq 6,\\[1mm]
[e_i,e_{11-i}]=(-1)^ie_{11},&2\leq i \leq 5,\\[1mm]
\end{array}\right.\]
where $\{e_1,e_2,\dots,e_{n}\}$ is a basis of the algebra.
\end{thm}

We call a vector space $M$ a {\it module over a Leibniz algebra} $L$ if there are two bilinear maps:
$$[-,-] \colon L\times M \rightarrow M \qquad \text{and} \qquad [-,-] \colon M\times L \rightarrow M$$
satisfying the following three axioms
\begin{align*}
[m,[x,y]] & =[[m,x],y]-[[m,y],x],\\
[x,[m,y]] & =[[x,m],y]-[[x,y],m],\\
[x,[y,m]] & =[[x,y],m]-[[x,m],y],
\end{align*}
for any $m\in M$, $x, y \in L$.

For a Leibniz algebra $L$ and module $M$ over $L$ we consider the spaces
$$CL^0(L,M) = M, \quad CL^n(L,M)=\Hom(L^{\otimes n}, M), \ n > 0.$$

Let $d^n : CL^n(L,M) \rightarrow CL^{n+1}(L,M)$ be an
$F$-homomorphism defined by
 \begin{multline*}
(d^n\varphi)(x_1, \dots , x_{n+1}): = [x_1,\varphi(x_2,\dots,x_{n+1})]
+\sum\limits_{i=2}^{n+1}(-1)^{i}[\varphi(x_1,
\dots, \widehat{x}_i, \dots , x_{n+1}),x_i]\\
+\sum\limits_{1\leq i<j\leq {n+1}}(-1)^{j+1}\varphi(x_1, \dots,
x_{i-1},[x_i,x_j], x_{i+1}, \dots , \widehat{x}_j, \dots
,x_{n+1}),
\end{multline*}
 where $\varphi\in CL^n(L,M)$ and $x_i\in L$. The property $d^{n+1}d^n=0$ leads that the derivative
operator $d=\sum\limits_{i \geq 0}d^i$ satisfies the property
$d\circ d = 0$. Therefore, the $n$-th cohomology group is well defined by
$$HL^n(L,M): = ZL^n(L,M)/ BL^n(L,M),$$
where the elements $ZL^n(L,M):=\Ker d^{n+1}$ and $BL^n(L,M):=\Img d^n$ are called {\it
$n$-cocycles} and {\it $n$-coboundaries}, respectively.

Although in general the computation of cohomology  groups is complicated, the
Hochschild-Serre factorization theorem simplifies its computation for semidirect sums
of algebras \cite{Leger}. If $R=N\dot{+}Q$ is a solvable Lie algebra such that $Q$ is abelian and the operators $R_x$ for all $x\in Q$ are diagonal.
Then the adjoint Lie $n$-th cohomology $H^n(R,R)$ satisfies the following isomorphism \cite{AnCaGa3,AnCaGa4}

\begin{equation}\label{eq1}
H^n(R,R)\simeq \sum\limits_{a+b=n}H^a(Q,F)\otimes H^b(N,R)^Q,
\end{equation}
where
\begin{equation} \label{eq2}
 H^b(N,R)^Q=\{\varphi\in H^b(N,R) \mid (d^b\varphi)(x,z_1,\dots,z_b)=0, \ x\in Q, \ z_i\in N\}.
 \end{equation}

The linear reductive group $\GL_n(\mathbb{F})$ acts on $\Leib_n$, the variety of $n$-dimensional Leibniz algebra structures, via change of basis, i.e.
\[(g*\lambda)(x,y)=g \Big(\lambda \big(g^{-1}(x),g^{-1}(y) \big) \Big), \quad  g \in \GL_n(\mathbb{F}), \  \lambda \in \Leib_n.\]

The orbits $\Orb(-)$ under this action are the isomorphism classes of algebras. Recall, Leibniz algebras with open orbits are called \emph{rigid}. The importance of rigid algebras follows from the following fact: the closure of orbit of rigid algebra forms an irreducible component of the variety and the variety is the union of finite numbers of irreducible components.

Here we give a result, which asserts that triviality of the second cohomology of a Leibniz algebra implies its rigidity.

\begin{thm}[\cite{Balavoine}] \label{Bal}
If the second cohomology of a Leibniz algebra $HL^2(L,L)=\{0\}$, then it is a rigid algebra.
\end{thm}

\section{Solvable Leibniz algebras with quasi-filiform Lie algebras of maximum length nilradicals.}

In this section we describe solvable Leibniz algebras whose nilradical is a quasi-filiform Lie algebra of maximum length.
 The classification of solvable Leibniz algebras with five-dimensional quasi-filiform Lie algebras of maximal length already obtained in \cite{aloberdi}. Therefore, we   focus to the case of nilradical is a quasi-filiform Lie algebra of maximum length of dimension greater than five.

For a solvable Leibniz algebra $R=N\oplus Q$ with nilradical $N$ and $s$-dimensional complemented space $Q$ we   use notation $R(N,s)$.

In order to start the description we need to know the derivations of quasi-filiform Lie algebras of maximum length.

\subsection{Derivations of algebras $g^i, i=1, 2, 3$}

\

\

In this subsection we describe the spaces of derivations of quasi-filiform Lie algebras of maximum length.

\begin{prop} \label{prop1}
Any derivation of quasi-filiform Lie algebras of maximum length has the following form:

for the algebra $g^1_{(n,1)}$,

$$\begin{cases}
d(e_1)=\sum\limits_{t=1}^na_{t}e_t, \\[1mm]
d(e_2)=\sum\limits_{t=2}^{n}b_{t}e_t, \\[1mm]
d(e_i)=((i-2)a_{1}+b_{2})e_i+\sum\limits_{t=i+1}^{n-1}b_{t-i+2}e_t+(-1)^i a_{n-i+1}e_n,\ 3\leq i\leq n-1,\\[1mm]
d(e_n)=((n-4)a_{1}+2b_{2})e_n, \\[1mm]
\end{cases}$$
where $b_{2k}=0,\ 2\leq k\leq \frac{n-3}{2};$

for the algebra $g^2_{(n,1)}$,

$$\begin{cases}
d(e_1)=\sum\limits_{t=1}^{n-1}a_{t}e_t, \\[1mm]
d(e_i)=((i-2)a_{1}+b_{2})e_i+\sum\limits_{t=i+1}^{n-1}b_{t-i+2}e_t,\ 2\leq i\leq n-1,\\[1mm]
d(e_n)=-\sum\limits_{t=3}^{n-2}a_{t-1}e_t+c_{n-1}e_{n-1}+2a_{1}e_{n}, \\[1mm]
\end{cases}$$

for the algebra $g^3_{(n,1)}$,
$$\left\{\begin{array}{ll}
d(e_1)=a_1e_1+\sum\limits_{t=3}^{n-1}a_te_t,\\[1mm]
d(e_2)=3a_1e_2+\sum\limits_{t=3}^{n-1}b_te_t,\\[1mm]
d(e_i)=(i+1)a_1e_i+b_3e_{i+1}+b_4e_{i+2}+\sum\limits_{j=i+3}^{n-1}(b_{j-i+2}-a_{j-i})e_{j},&3\leq i\leq n-1,\\[1mm]
d(e_n)=-\sum\limits_{t=4}^{n-2}a_{t-1}e_t+c_{n-1}e_{n-1}+2a_1e_n,\\[1mm]
\end{array}\right.$$

for the algebra $g^1_{7}$,
$$\left\{\begin{array}{ll}
d(e_1)=a_{1}e_1+a_{2}e_2+a_{3}e_3+a_{4}e_4+a_{5}e_5+a_{6}e_6+a_{7}e_7,\\[1mm]
d(e_2)=2a_{1}e_2+b_{3}e_3+b_{5}e_5+b_{6}e_6+b_{7}e_7,\\[1mm]
d(e_3)=3a_{1}e_3+b_{3}e_4-a_{3}e_5+(b_{5}-a_{4})e_6-a_{5}e_7,\\[1mm]
d(e_4)=4a_{1}e_4+b_{3}e_5-a_{3}e_6+a_{4}e_7,\\[1mm]
d(e_5)=5a_{1}e_5+b_{3}e_6-a_{3}e_7,\\[1mm]
d(e_6)=6a_{1}e_6,\\[1mm]
d(e_7)=7a_{1}e_7,\\[1mm]
\end{array}\right.$$

for the algebra $g^2_{9}$,

$$\left\{\begin{array}{ll}
d(e_1)=a_{1}e_1+a_{3}e_3+a_{4}e_4+a_{5}e_5+a_{6}e_6+a_{7}e_7+a_{8}e_8+a_{9}e_9,\\[1mm]
d(e_2)=2a_{1}e_2+b_{3}e_3+a_{4}e_5+a_{5}e_6+b_{7}e_7+b_{8}e_8+b_{9}e_9,\\[1mm]
d(e_3)=3a_{1}e_3+b_{3}e_4-a_{3}e_5-2a_{5}e_7+(b_{7}-5a_{6})e_8-a_{7}e_9,\\[1mm]
d(e_4)=4a_{1}e_4+b_{3}e_5-a_{3}e_6+2a_{4}e_7+a_{6}e_9,\\[1mm]
d(e_5)=5a_{1}e_5+b_{3}e_6-3a_{3}e_7+2a_{4}e_8-a_{5}e_9,\\[1mm]
d(e_6)=6a_{1}e_6+b_{3}e_7-5a_{3}e_8+a_{4}e_9,\\[1mm]
d(e_7)=7a_{1}e_7+b_{3}e_8-a_{3}e_9,\\[1mm]
d(e_8)=8a_{1}e_8,\\[1mm]
d(e_9)=9a_{1}e_9,\\[1mm]
\end{array}\right.$$

for the algebra $g^3_{11}$,

$$\left\{\begin{array}{ll}
d(e_1)=a_{1}e_1+a_{3}e_3+a_{4}e_4+a_{5}e_5+a_{6}e_6+a_{7}e_7+a_{8}e_8+a_{9}e_9+a_{10}e_{10}+a_{11}e_{11},\\[1mm]
d(e_2)=2a_{1}e_2+b_{3}e_3+a_{4}e_5+a_{5}e_6+b_{7}e_7-a_{7}e_8+b_{9}e_9+b_{10}e_{10}+b_{11}e_{11},\\[1mm]
d(e_3)=3a_{1}e_3+b_{3}e_4-a_{3}e_5+a_{5}e_7+(b_{7}+a_{6})e_8-b_{9}e_9-a_{9}e_{11},\\[1mm]
d(e_4)=4a_{1}e_4+b_{3}e_5-a_{3}e_6-a_{4}e_7+(a_{6}+b_{7})e_9+a_{7}e_{10}+a_{8}e_{11},\\[1mm]
d(e_5)=5a_{1}e_5+b_{3}e_6-a_{4}e_8-a_{5}e_9+b_{7}e_{10}-a_{7}e_{11},\\[1mm]
d(e_6)=6a_{1}e_6+b_{3}e_7+a_{3}e_8-a_{5}e_{10}+a_{6}e_{11},\\[1mm]
d(e_7)=7a_{1}e_7+b_{3}e_8+a_{3}e_9+a_{4}e_{10}-a_{5}e_{11},\\[1mm]
d(e_8)=8a_{1}e_8+b_{3}e_9+a_{4}e_{11},\\[1mm]
d(e_9)=9a_{1}e_9+b_{3}e_{10}-a_{3}e_{11},\\[1mm]
d(e_{10})=10a_{1}e_{10},\\[1mm]
d(e_{11})=11a_{1}e_{11},\\[1mm]
\end{array}\right.$$

\end{prop}
\begin{proof}
From Theorem \ref{thm26} we conclude that $e_1$ and $e_2$ are the generator basis elements of the algebra $g^{1}_{(n,1)}$.

We put
\[ d(e_1)=\sum\limits_{t=1}^{n}a_te_t,, \qquad d(e_{2})=\sum\limits_{t=1}^{n}b_te_t.\]

From the derivation property \eqref{eq0} we have

$$d(e_3)=d([e_1,e_2])=[d(e_1),e_2]+[e_1,d(e_2)]=(a_1+b_2)e_3+\sum\limits_{t=4}^{n-1}b_{t-1}e_t-a_{n-2}e_n.$$

By induction and the derivation property, we derive
$$d(e_i)=[(i-2)a_1+b_2]e_i+\sum\limits_{t=i+1}^{n-1}b_{t-i+2}e_t+(-1)^i a_{n-i+1}e_n,$$
where $3\leq i\leq n-1$.

From $0=d([e_2,e_i])=[d(e_2),e_i]+[e_2,d(e_i)], \ 3\leq i\leq n-3$, we conclude
$$b_1=0, \ (-1)^{i}b_{n-i}-b_{n-i}=0, \ \ 3\leq i\leq n-3.$$

Consequently,
$$b_{2k}=0, \ \ 2\leq k\leq \frac{n-3}{2}.$$

The equality $d(e_n)=d([e_2,e_{n-2}])$ implies
\[d(e_n)=d([e_2,e_{n-2}])=((n-4)a_1+2b_2)e_n.\]

The descriptions of derivations for other algebras are obtained by applying  similar calculations used for the algebra $g^{1}_{(n,1)}$.
\end{proof}

\subsection{Descriptions of algebras $R(g^{i}, 1), \ i=1, 2, 3$}

\

\

In this subsection we classify solvable Leibniz algebras with quasi-filiform Lie nilradicals of maximum length under the condition that the complemented space is one-dimensional.

\begin{thm} \label{thmmu1}
 An arbitrary algebra of the family $R(g^1_{(n,1)},1)$ admits a basis $\{e_1, e_2, \dots, e_{n},x\}$ such that its multiplications table has one of the following types:
$$R_1(g^1_{(n,1)},1)(a_2,b_5,\dots,b_{n-1}):\begin{cases}
[e_1,x]=-[x,e_1]=a_{2}e_2,\\[1mm]
[e_i,x]=-[x,e_i]=e_i+\sum\limits_{t=i+2}^{n-1}b_{t-i+2}e_t,\ 2\leq i\leq n-2,\\[1mm]
[e_{n-1},x]=-[x,e_{n-1}]=e_{n-1}+a_2e_n,\\[1mm]
[e_n,x]=-[x,e_n]=2e_n, \ \ \text{where} \ \  b_{2k}=0,\ 2\leq k\leq \frac{n-3}{2},\\[1mm]
\end{cases}$$
$$R_2(g^1_{(n,1)},1)(a_{n-1}):\begin{cases}
[e_1,x]=-[x,e_1]=e_1+a_{n-1}e_{n-1}, \\[1mm]
[e_i,x]=-[x,e_i]=(i+2-n)e_i,& 2\leq i\leq n-1,\\[1mm]
[e_n,x]=-[x,e_n]=(4-n)e_n, \\[1mm]
\end{cases}$$
$$R_3(g^1_{(n,1)},1)(\delta_{n-1}):\begin{cases}
[e_1,x]=-[x,e_1]=e_1, \\[1mm]
[e_i,x]=-[x,e_i]=(i+1-n)e_i,& 2\leq i\leq n-2,\\[1mm]
[e_n,x]=-[x,e_n]=(2-n)e_n, \\[1mm]
[x,x]=\delta_{n-1}e_{n-1},\\[1mm]
\end{cases}$$
$$R_4(g^1_{(n,1)},1)(\delta_{n}):\begin{cases}
[e_1,x]=-[x,e_1]=e_1, \\[1mm]
[e_i,x]=-[x,e_i]=(i-\frac{n}{2})e_i,& 2\leq i\leq n-1,\\[1mm]
[x,x]=\delta_{n}e_{n},\\[1mm]
\end{cases}$$
$$R_5(g^1_{(n,1)},1)(a_{n}):\begin{cases}
[e_1,x]=-[x,e_1]=e_1+a_ne_n, \\[1mm]
[e_i,x]=-[x,e_i]=(i+\frac{1-n}{2})e_i,\ &2\leq i\leq n-1,\\[1mm]
[e_n,x]=-[x,e_n]=e_n, \\[1mm]
\end{cases}$$
$$R_6(g^1_{(n,1)},1)(a_{2}):\begin{cases}
[e_1,x]=-[x,e_1]=e_1+a_{2}e_2, \\[1mm]
[e_i,x]=-[x,e_i]=(i-1)e_i,\ &2\leq i\leq n-2,\\[1mm]
[e_{n-1},x]=-[x,e_{n-1}]=(n-2)e_{n-1}+a_2e_n,\\[1mm]
[e_n,x]=-[x,e_n]=(n-2)e_n, \\[1mm]
\end{cases}$$
$$R_7(g^1_{(n,1)},1)(b_{2}):\begin{cases}
[e_1,x]=-[x,e_1]=e_1, \\[1mm]
[e_i,x]=-[x,e_i]=(i-2+b_{2})e_i,\ 2\leq i\leq n-1,\\[1mm]
[e_n,x]=-[x,e_n]=(n-4+2b_{2})e_n, \ b_2\notin \{4-n,\ 3-n,\ \frac{4-n}{2},\
\frac{5-n}{2},\ 1\}.\\[1mm]
\end{cases}$$
\end{thm}
\begin{proof} Let $\{e_1, e_2, \dots, e_{n}\}$ be a basis of the algebra $g^1_{(n,1)}$ such that the table of multiplication in this basis has the form of Theorem \ref{thm26}. We know that the operator $\mathcal{R}_{{x |}_{g^1_{(n,1)}}}$ for $x\in Q$ is an non-nilpotent derivation of $g^1_{(n,1)}$. Taking into account the structure of the algebra $g^1_{(n,1)}$ we conclude
$\{e_{1},\dots, e_{n-2}\}\bigcap \Ann_r(R(g^1_{(n,1)}, 1))=\emptyset$. Moreover, we have $[x,e_i]+[e_i,x],\ [x,x] \in \Ann_r(R(g^1_{(n,1)}, 1))$. Now using the description of derivations of the algebra $g^1_{(n,1)}$ from Proposition \ref{prop1} we get the following products in $R(g^1_{(n,1)}, 1)$:
$$\begin{cases}
[e_1,x]=\sum\limits_{t=1}^na_{t}e_t, \\[1mm]
[e_2,x]=\sum\limits_{t=2}^{n}b_{t}e_t, \\[1mm]
[e_i,x]=((i-2)a_{1}+b_{2})e_i+\sum\limits_{t=i+1}^{n-1}b_{t-i+2}e_t+(-1)^i a_{n-i+1}e_n,\ 3\leq i\leq n-1,\\[1mm]
[e_n,x]=((n-4)a_{1}+2b_{2})e_n, \\[1mm]
[x,e_1]=-\sum\limits_{t=1}^{n-2}a_{t}e_t+\mu_{1,n-1}e_{n-1}+\mu_{1,n}e_{n},\\[1mm]
[x,e_i]=-((i-2)a_{1}+b_{2})e_i-\sum\limits_{t=i+1}^{n-2}b_{t-i+2}e_t+\mu_{i,n-1}e_{n-1}+\mu_{i,n}e_{n},\ 2\leq i\leq n-2,\\[1mm]
[x,e_{n-1}]=\mu_{n-1,n-1}e_{n-1}+\mu_{n-1,n}e_{n},\\[1mm]
[x,e_n]=\mu_{n,n-1}e_{n-1}+\mu_{n,n}e_{n},\\[1mm]
[x,x]=\delta_{n-1}e_{n-1}+\delta_{n}e_{n}.\\[1mm]
\end{cases}$$

Consider the Leibniz identity for the following triples of elements to deduce
$$\left\{\begin{array}{lll}
{\mathcal L}(x,e_1,e_{n-2})=0 & \Rightarrow & [x,e_{n-1}]=-((n-3)a_{1}+b_{2})e_{n-1}-a_2e_n,\\[1mm]
{\mathcal L}(x,e_2,e_{n-2})=0, & \Rightarrow & [x,e_{n}]=-((n-4)a_{1}+2b_{2})e_{n}, \\[1mm]
{\mathcal L}(x,x,x)=0, & \Rightarrow & ((n-3)a_{1}+b_{2})\delta_{n-1}=0,\ \ a_2\delta_{n-1}+((n-4)a_{1}+2b_{2})\delta_{n}=0. \\[1mm]
\end{array}\right.$$

The equalities ${\mathcal L}(x,e_1,e_i)={\mathcal L}(x,e_2,x)={\mathcal L}(x,x,e_2)={\mathcal L}(x,e_1,x)={\mathcal L}(x,x,e_1)=0$ for $2\leq i\leq n-2$ imply $[e_i,x]=-[x,e_i],\ 1\leq i\leq n$.

Thus, we obtain the following table of multiplications of the algebra $R(g^1_{(n,1)}, 1)$:
$$\begin{cases}
[e_1,x]=\sum\limits_{t=1}^na_{t}e_t, \\[1mm]
[e_2,x]=\sum\limits_{t=2}^{n}b_{t}e_t, \\[1mm]
[e_i,x]=((i-2)a_{1}+b_{2})e_i+\sum\limits_{t=i+1}^{n-1}b_{t-i+2}e_t+(-1)^i\alpha_{n-i+1}e_n,\ 3\leq i\leq n-1,\\[1mm]
[e_n,x]=((n-4)a_{1}+2b_{2})e_n, \\[1mm]
[x,e_1]=-\sum\limits_{t=1}^na_{t}e_t, \\[1mm]
[x,e_2]=-\sum\limits_{t=2}^{n}b_{t}e_t, \\[1mm]
[x,e_i]=-((i-2)a_{1}+b_{2})e_i-\sum\limits_{t=i+1}^{n-1}b_{t-i+2}e_t-(-1)^i a_{n-i+1}e_n,\ 3\leq i\leq n-1,\\[1mm]
[x,e_n]=-((n-4)a_{1}+2b_{2})e_n,\\[1mm]
[x,x]=\delta_{n-1}e_{n-1}+\delta_{n}e_{n}.\\[1mm]
\end{cases}$$
where $b_{2k}=0,\ 2\leq k\leq \frac{n-3}{2}$ and
\begin{equation}\label{restrictions}
((n-3)a_{1}+b_{2})\delta_{n-1}=a_2\delta_{n-1}+((n-4)a_{1}+2b_{2})\delta_{n}=0.
\end{equation}

Let us take the general change of generators basis elements:
$$e_1^\prime=\sum\limits_{i=1}^{n}A_{i}e_i,\ \ \ e_2'=\sum\limits_{t=1}^{n}B_{t}e_t,\ \ \ x^\prime=Hx+\sum\limits_{t=1}^{n}C_{t}e_t.$$

Consider the table of multiplications of the algebra $R(g^1_{(n,1)}, 1)$ in the new basis and then express those products in terms of the old basis to obtain the relations:
$$a_1'=Ha_1,\ \ \ b_{2}'=Hb_2.$$
It is known that $(a_1, b_2)\neq(0,0)$, otherwise the algebra $R(g^1_{(n,1)}, 1)$ is nilpotent algebra.

Now we need to consider the following cases.

{\bf Case 1.} Let $a_1=0$. Then by choosing appropriate value of $H$ one can assume $b_2=1$ and from (\ref{restrictions}) we have $\delta_{n-1}=\delta_{n}=0$.

Taking the change of generator basis elements in the following form:
$$e_1'=e_1-\frac{a_2(b_n-a_{n-1})+a_n}{2}e_n,\ e_2'=e_2-(b_n-a_{n-1})e_n,\ x'=x+b_3e_1-\sum_{i=3}^{n-1}a_ie_{i-1},$$
we can assume $b_3=b_n=a_t=0,\ 3\leq t\leq n$. Thus, the algebra $R_1(g^1_{(n,1)},1)(a_2,b_4,\dots,b_{n-1})$ is obtained.

{\bf Case 2.} Let $a_1\neq0$. Then for suitable value of $H$ one can assume $a_1=1$ and from (\ref{restrictions}) we derive
$$(n-3+b_{2})\delta_{n-1}=0,\ a_2\delta_{n-1}+(n-4+2b_{2})\delta_{n}=0.$$

The change of generator basis elements of the form
$$e_1^\prime=e_1,\ \ e_i^\prime=e_i+\sum\limits_{j=i+1}^{n-1}A_{j-i+2}e_{j}, \ 2\leq i\leq n-1,$$
with $$A_{3}=-b_{3},\ A_{i}=-\frac{1}{i-2}\Big(b_{i}+\sum\limits_{j=3}^{i-1}b_{i-j+2}A_{j}\Big),4\leq i\leq n-1.$$
lead to $b_t=0,\ 3\leq t\leq n-1$.

Setting $x'=x-\sum_{i=3}^{n-2}a_ie_{i-1}-b_ne_{n-2}$, we derive $b_n=a_t=0,\ 3\leq t\leq n-2$.

Thus, the multiplication table of the algebra $R(g^1_{(n,1)},1)$ has the form:
\begin{equation}\label{tablemult}
\begin{cases}
[e_1,x]=e_1+a_{2}e_2+a_{n-1}e_{n-1}+a_ne_n, \\[1mm]
[e_i,x]=(i-2+b_{2})e_i,\ &2\leq i\leq n-2,\\[1mm]
[e_{n-1},x]=(n-3+b_{2})e_{n-1}+a_2e_n,\\[1mm]
[e_n,x]=(n-4+2b_{2})e_n, \\[1mm]
[x,e_1]=-e_1-a_{2}e_2-a_{n-1}e_{n-1}-a_ne_n, \\[1mm]
[x,e_i]=-(i-2+b_{2})e_i,\ &2\leq i\leq n-2,\\[1mm]
[x,e_{n-1}]=-(n-3+b_{2})e_{n-1}-a_2e_n,\\[1mm]
[x,e_n]=-(n-4+2b_{2})e_n,\\[1mm]
[x,x]=\delta_{n-1}e_{n-1}+\delta_{n}e_{n},\\[1mm]
\end{cases}
\end{equation}
with
\begin{equation}\label{restrictions1}
(n-3+b_{2})\delta_{n-1}=0,\ a_2\delta_{n-1}+(n-4+2b_{2})\delta_{n}=0.
\end{equation}

Applying the following change of generator basis elements $e_1$ and $e_{n-1}$ in (\ref{tablemult}):
$$e_1'=e_1+Ae_{2}+Be_{n-1}+Ce_{n},\ e_{n-1}'=e_{n-1}+Ae_{n},$$ we derive the restrictions
\begin{equation}\label{restrictions2}
a_{2}'=a_{2}+A(b_2-1),\ a_{n-1}'=a_{n-1}+B(n-4+b_2),\ a_{n}'=a_{n}+Ba_2+C(n-5+2b_2)-a_{n-1}'A.
\end{equation}

Now we   study the following possible subcases.

{\bf Subcase 2.1.} Let $b_2=4-n$. Then from (\ref{restrictions1}) we get $\delta_{n-1}=\delta_n=0$ and choosing $A=\frac{a_{2}}{n-3},\ C=\frac{a_{n}}{n-3}-\frac{a_{n-1}a_2}{(n-3)^2}$ in (\ref{restrictions2})
we can assume $a_{2}=a_{n}=0$. Hence, the algebra $R_2(g^1_{(n,1)},1)(a_{n-1})$ is obtained.

{\bf Subcase 2.2.} Let $b_2=3-n$. Then setting $A=\frac{a_{2}}{n-2}, \ B=a_{n-1},\ C=\frac{a_{n}+a_{2}a_{n-1}}{n-1}$ in (\ref{restrictions2}) we obtain $a_{2}=a_{n-1}=a_n=0$. The restriction (\ref{restrictions1}) lead to $\delta_n=0$. So, we get the algebra $R_3(g^1_{(n,1)},1)(\delta_{n-1})$.

{\bf Subcase 2.3.} Let $b_2=\frac{4-n}{2}$. Then from (\ref{restrictions1}) we deduce $\delta_{n-1}=0$ and choosing
$A=\frac{2a_{2}}{n-2},\ B=-\frac{2a_{n-1}}{n-4}, \ C=a_{n}-\frac{2a_2a_{n-1}}{n-4}$ in (\ref{restrictions2}) we get $a_2=a_{n-1}=a_{n}=0$. Hence, we obtain the algebra $R_4(g^1_{(n,1)},1)(\delta_{n})$.

{\bf Subcase 2.4.} Let $b_2=\frac{5-n}{2}$. Then restrictions (\ref{restrictions1}) imply $\delta_{n-1}=\delta_n=0$ and putting $A=\frac{2a_{2}}{n-3},\ B=-\frac{2a_{n-1}}{n-3}$ in (\ref{restrictions2}) we have $a_{2}=a_{n-1}=0$. Therefore, we get the algebra $R_5(g^1_{(n,1)},1)(a_{n})$.

{\bf Subcase 2.5.} Let $b_2=1$. Then we have $\delta_{n-1}=\delta_n=0$ and taking $B=-\frac{a_{n-1}}{n-3},\ C=-\frac{a_{n}}{n-3}+\frac{a_{n-1}a_2}{(n-3)^2}$ in (\ref{restrictions2}) one can assume $a_{n-1}=a_{n}=0$. Hence, we derive the algebra $R_6(g^1_{(n,1)},1)(a_{2})$.

{\bf Subcase 2.6.} Finally, let $b_2\neq 4-n,\ 3-n,\ \frac{4-n}{2},\
\frac{5-n}{2},\ 1$. Then from (\ref{restrictions1}) we have $\delta_{n-1}=\delta_n=0$ and choosing
$A=-\frac{a_{2}}{b_2-1},\ B=-\frac{a_{n-1}}{n-4+b_2},\ C=-\frac{a_{n}}{n-5+2b_2}+\frac{a_2a_{n-1}}{(n-5+2b_2)(n-4+b_2)}$ in (\ref{restrictions2}) one can assume $a_2=a_{n-1}=a_n=0$. Thus, we obtain the algebra $R_7(g^1_{(n,1)},1)(b_{2})$.
\end{proof}

In the next theorem we investigate the isomorphisms inside each family of algebras of Theorem \ref{thmmu1}.

\begin{thm} An arbitrary algebra of the family $R(g^1_{(n,1)},1)$ is isomorphic to one of the following pairwise non-isomorphic algebras:

$$R_1(g^1_{(n,1)},1)(a_2,b_4,\dots,b_{n-1}), \ R_2(g^1_{(n,1)},1)(1), \ R_3(g^1_{(n,1)},1)(1),\ R_4(g^1_{(n,1)},1)(1),$$$$ R_5(g^1_{(n,1)},1)(1),\ R_6(g^1_{(n,1)},1)(1),\
R_7(g^1_{(n,1)},1)(b_{2}).$$

Note that the first non-vanishing parameter $\{a_2,b_4,\dots,b_{n-1}\}$
in the algebra  $R_1(g^1_{(n,1)},1)(a_2,b_4,\dots,b_{n-1})$ can be scaled to $1$.

\end{thm}

\begin{proof} We start with the general change of generator basis elements of the algebra $R(g^1_{(n,1)},1)$:
\[e_1^\prime=\sum\limits_{t=1}^{n}A_{t}e_t, \quad e_2^\prime=\sum\limits_{t=1}^{n}B_{t}e_t,\quad x^\prime=Hx+\sum\limits_{t=1}^{n}C_{t}e_t.\]

Write the multiplications table of an algebra from the considered family in the new basis $\{e_1', \dots, e_n',  x'\}$ and express them via basis $\{e_1, \dots, e_n,  x\}$ to get invariant relations on the structure constants of the algebra.

\begin{itemize}
\item For the algebra $R_1(g^1_{(n,1)},1)(a_2,b_4,\dots,b_{n-1})$, we have the following invariant relations:

$$a_{2}'=\frac{A_1a_2}{B_2}, \ b_{t}'=\frac{b_t}{A_1^{t-2}},\ \ 4\leq t\leq n-1, \ \text{where} \ \  b_{2k}=0,\ 2\leq k\leq \frac{n-3}{2}.$$

Note that the first non-vanishing parameter $\{a_2,b_4,\dots,b_{n-1}\}$
in the algebra  can be scaled to $1$.

\item  Consider the family of algebras $R_2(g^1_{(n,1)},1)(a_{n-1})$. Then we have the following invariant relation:

$$a_{n-1}'=\frac{a_{n-1}}{A_1^{n-4}B_2}.$$

If $a_{n-1}\neq0$, then by putting $B_2=\frac{a_{n-1}}{A_1^{n-4}}$, we get $a_{n-1}'=1$. Hence, the algebra $R_2(g^1_{(n,1)},1)(1)$ is obtained;

If $a_{n-1}=0$, then we obtain $R_7(g^1_{(n,1)},1)(4-n)$.

\item Consider the family of algebras $R_3(g^1_{(n,1)},1)(\delta_{n-1})$.  Then the invariant relation is the following
\[\delta_{n-1}'=\frac{\delta_{n-1}}{A_{1}^{n-3}B_2}.\]

If $\delta_{n-1}\neq0$, then taking $B_2=\frac{\delta_{n-1}}{A_{1}^{n-3}}$, we can assume $\delta_{n-1}'=1$ and we obtain $R_3(g^1_{(n,1)},1)(1)$;

If $\delta_{n-1}=0$, then we derive $R_7(g^1_{(n,1)},1)(3-n)$.

\item Let consider the family $R_4(g^1_{(n,1)},1)(\delta_{n})$.  Then
\[\delta_{n}'=\frac{\delta_{n}}{A_{1}^{n-4}B_2^2}.\]

If $\delta_{n}\neq0$, then taking $B_2=\sqrt{\frac{\delta_{n}}{A_{1}^{n-4}}}$, we can assume $\delta_{n}'=1$ and we obtain $R_4(g^1_{(n,1)},1)(1)$;

If $\delta_{n}=0$, we derive $R_7(g^1_{(n,1)},1)(\frac{4-n}{2})$.

\item  For the family of algebras $R_5(g^1_{(n,1)},1)(a_{n})$ we have the following invariant relation:
$$a_{n}'=\frac{a_{n}}{A_1^{n-5}B_2^2}.$$

If $a_{n}\neq0$, then by putting $B_2=\sqrt{\frac{a_{n}}{A_1^{n-5}}}$, we get $a_{n}'=1$ and $R_5(g^1_{(n,1)},1)(1)$;

If $a_{n}=0$, then we obtain $R_7(g^1_{(n,1)},1)(\frac{5-n}{2})$.

\item Consider the family $R_6(g^1_{(n,1)},1)(a_2)$. Then we obtain
\[a_{2}'=\frac{a_{2}A_1}{B_{2}}.\]

If $a_{2}\neq0$, then putting $B_2=A_1a_2$ we derive $a_2'=1$ and hence, we obtain the algebra $R_6(g^1_{(n,1)},1)(1)$;

If $a_{2}=0$, then we derive $R_7(g^1_{(n,1)},1)(1)$.

\item Considering the family of algebras $R_7(g^1_{(n,1)},1)(b_2)$ we have $b_{2}'=b_2$, which show that all algebras of this family are pairwise non-isomorphic.
\end{itemize}
\end{proof}

The next results are established in a similar way as for $R(g^2_{(n,1)},1)$.

\begin{thm} An arbitrary algebra of the family $R(g^2_{(n,1)},1) $ admits a basis such that its multiplications table has one of the following types:
$$R_1(g^2_{(n,1)},1)(b_4,\dots,b_{n-1}): \ [e_i,x]=-[x,e_i]=e_i+\sum\limits_{t=i+2}^{n-1}b_{t-i+2}e_t,\ 2\leq i\leq n-1,$$
$$R_2(g^2_{(n,1)},1)(a_{2}):\begin{cases}
[e_1,x]=-[x,e_1]=e_1+a_{2}e_2,\\[1mm]
[e_i,x]=-[x,e_i]=(i-1)e_i,\ 2\leq i\leq n-1,\\[1mm]
[e_n,x]=-[x,e_n]=-a_{2}e_3+2e_{n},\\[1mm]
\end{cases}$$

$$R_3(g^2_{(n,1)},1)(\gamma_{n-1}):\begin{cases}
[e_1,x]=-[x,e_1]=e_1,\\[1mm]
[e_i,x]=-[x,e_i]=(i+3-n)e_i,\ 2\leq i\leq n-1,\\[1mm]
[e_n,x]=-[x,e_n]=\gamma_{n-1}e_{n-1}+2e_{n},\\[1mm]
\end{cases}$$

$$R_4(g^2_{(n,1)},1)(\delta_{n-1}):\begin{cases}
[e_1,x]=-[x,e_1]=e_1,\\[1mm]
[e_i,x]=-[x,e_i]=(i+1-n)e_i,\ 2\leq i\leq n-1,\\[1mm]
[e_n,x]=-[x,e_n]=2e_{n},\ \ [x,x]=\delta_{n-1}e_{n-1}.\\[1mm]
\end{cases}$$

$$R_5(g^2_{(n,1)},1)(b_{2}):\begin{cases}
[e_1,x]=-[x,e_1]=e_1,\\[1mm]
[e_i,x]=-[x,e_i]=(i-2+b_{2})e_i,\ 2\leq i\leq n-1,\\[1mm]
[e_n,x]=-[x,e_n]=2e_{n},\ \ b_2\neq 1,\ 5-n,\ 3-n.\\[1mm]
\end{cases}$$
\end{thm}

\begin{thm}
An arbitrary algebra of the family $R(g^2_{(n,1)},1)$ is isomorphic to one of the following pairwise non isomorphic algebras:

$$R_1(g^2_{(n,1)},1)(b_4,\dots,b_{n-1}), \ R_3(g^2_{(n,1)},1)(1), \ R_4(g^2_{(n,1)},1)(1),\ R_5(g^2_{(n,1)},1)(b_{2}), \ b_2\in \mathbb{C}.$$

Note that the first non-vanishing parameter $\{b_4,\dots,b_{n-1}\}$
in the algebra  $R_1(g^2_{(n,1)},1)(b_4,\dots,b_{n-1})$ can be scaled to $1$.

\end{thm}

Analogously we obtain the descriptions of solvable algebras with nilradicals $g^3_{(n,1)}, g^1_{7}, g^2_{9}, g^3_{11}$ under the condition that complemented space is one-dimensional.

\begin{thm}\label{thmmu3} An arbitrary solvable Leibniz algebra whose nilradical is one of the following: $g ^ 3_ {(n, 1)}, g ^ 1_ {7}, g ^ 2_ {9}, g ^ 3_ {11}$ and with one-dimensional complemented space is isomorphic to the following corresponding algebras:
$$R(g^3_{(n,1)},1):\left\{\begin{array}{ll}
[e_1,x]=-[x,e_1]=e_1,&\\[1mm]
[e_i,x]=-[x,e_i]=(i+1)e_i,&2\leq i \leq {n-1},\\[1mm]
[e_n,x]=-[x,e_n]=2e_n.\\[1mm]
\end{array}\right.$$
$$R(g^1_{7},1):\ [e_i,x]=-[x,e_i]=ie_i,\ \ 1\leq i\leq 7.$$
$$R(g^2_{9},1):\ [e_i,x]=-[x,e_i]=ie_i, \ \ 1\leq i\leq 9.$$
$$R(g^3_{11},1):\ [e_i,x]=-[x,e_i]=ie_i,\ \ 1\leq i\leq 11.$$
\end{thm}

\

\subsection{Descriptions of algebras $R(g^{i}_{(n,1)}, 2), \ i=1, 2$}

\

In this subsection we classify solvable  Leibniz algebras whose nilradical is a quasi-filiform Lie algebra of maximal length and dimension of the complemented space $Q$ is equal to two.

Due to Proposition \ref{prop1} we conclude that the dimensions of the complementary spaces only for nilradicals $g^{i}_{(n,1)},  \ i=1, 2$ can be equal to two.

\begin{thm} \label{thmmu2}
An arbitrary Leibniz algebra of the family $R(g^{i}_{(n,1)}, 2), \ i=1, 2$ is isomorphic to one of the following Lie algebras:

$$R(g^{1}_{(n,1)}, 2):\left\{\begin{array}{ll}
[e_1,x]=-[x,e_1]=e_1,&\\[1mm]
[e_i,x]=-[x,e_i]=(i-2)e_i,&2\leq i\leq n-1,\\[1mm]
[e_n,x]=-[x,e_n]=(n-4)e_n,&\\[1mm]
[e_i,y]=-[y,e_i]=e_i,& 2\leq i\leq n-1,\\[1mm]
[e_n,y]=-[y,e_n]=2e_n,\\[1mm]
\end{array}\right.$$

$$R(g^{2}_{(n,1)}, 2):\left\{\begin{array}{ll}
[e_1,x]=-[x,e_1]=e_1,&\\[1mm]
[e_i,x]=-[x,e_i]=(i-2)e_i,&2\leq i \leq {n-1},\\[1mm]
[e_n,x]=-[x,e_n]=2e_n,&\\[1mm]
[e_i,y]=-[y,e_i]=e_i,& 2\leq i \leq {n-1}.\\[1mm]
\end{array}\right.$$

\end{thm}

\begin{proof}
Consider the description of derivations for the algebra $g^{1}_{(n,1)}$ given in Proposition \ref{prop1}.
Since two parameters $a_1,b_2$ appear on the diagonal, it is clear that we can obtain two nil-independent
derivations, such as those corresponding to $(a_1,b_2)=(1,0)$, and $(0,1)$, denote them by $\mathcal{R}_{x}$ and $\mathcal{R}_{y}$, respectively.

We consider $R(g^{1}_{(n,1)},2)=\langle e_1,\dots,e_{n},x,y \rangle$, where $\mathcal{R}_{x},\mathcal{R}_{y}$ are by right multiplications on $x$ and $y$, respectively. It is known that the subspace $\spann\{e_1,\dots,e_n,x\}$ forms a subalgebra of the algebra $R(g^{1}_{(n,1)},2)$. Therefore this subalgebra is isomorphic to the algebra $R_7(g^{1}_{(n,1)},1)(0)$ in the list of Theorem \ref{thmmu1}. Then the multiplication table of algebra $R(g^{1}_{(n,1)},2)$
has the following form:
$$\left\{\begin{array}{ll}
[e_1,x]=-[x,e_1]=e_1,\\[1mm]
[e_i,x]=-[x,e_i]=(i-2)e_i,&2\leq i\leq n-1,\\[1mm]
[e_n,x]=-[x,e_n]=(n-4)e_n,\\[1mm]
[e_1,y]=\sum\limits_{t=2}^{n}a_te_t,\\[1mm]
[e_2,y]=e_2+\sum\limits_{t=3}^{n}b_te_t,\\[1mm]
[e_i,y]=e_i+\sum\limits_{j=i+1}^{n-1}b_{j-i+2}e_j+(-1)^ia_{n-i+1}e_n,&3\leq i\leq n-1,\\[1mm]
[e_n,y]=2e_n,\\[1mm]
[y,e_i]=\sum\limits_{t=1}^{n}\delta_{2,i}^te_t,&1\leq i\leq n,\\[1mm]
[x,y]=\sum\limits_{t=1}^{n}t_{1,2}^te_t,\\[1mm]
[y,x]=\sum\limits_{t=1}^{n}t_{2,1}^te_t,\\[1mm]
[y,y]=\sum\limits_{t=1}^{n}t_{2,2}^te_t.\\[1mm]
\end{array}\right.$$

Taking into account that $e_{1},\dots, e_{n} \notin \Ann_r(R(g^1_{(n,2)}, 1))$, we derive the products $[y,e_i]=-[e_i,y], \ [x,y]=-[y,x], \ [y,y]=0$ for $1\leq i\leq n$.
Considering the following change of generator basis elements:
$$e_1^\prime=e_1+(a_{n-1}-b_{n})e_{n-1},\ \ e_2^\prime=e_2+(a_{n-1}-b_{n})e_{n},\ \ x^\prime=x-(n-4)(a_{n-1}-b_{n})e_{n-2},$$
$$y^\prime=y+b_{3}e_1-\sum\limits_{t=3}^{n-2}a_{t}e_{t-1}-(2a_{n-1}-b_{n})e_{n-2},$$
we obtain
$$b_{3}=b_n=a_{t}=0, \quad 3\leq t\leq n-1.$$

Applying the Leibniz identity for triples of elements we obtain
$$\left\{\begin{array}{lll}
{\mathcal L}(e_2,y,x)=0 & \Rightarrow & t_{2,1}^1=t_{2,1}^{n-2}=b_{t}=0,\ 4\leq t\leq n-1, \\[1mm]
{\mathcal L}(e_1,y,x)=0 & \Rightarrow & t_{2,1}^{t}=a_{2}=a_{n}=0,\ 2\leq t\leq n-3.  \\[1mm]
\end{array}\right.$$

Putting
$$y^\prime=y-\frac{t_{2,1}^{n-1}}{n-3}e_{n-1}-\frac{t_{2,1}^n}{n-4}e_n,$$
we get $t_{2,1}^{n-1}=t_{2,1}^{n}=0$.

Thus we obtain the following algebra.
$$R(g^{1}_{(n,1)}, 2):\left\{\begin{array}{ll}
[e_1,x]=-[x,e_1]=e_1,&\\[1mm]
[e_i,x]=-[x,e_i]=(i-2)e_i,&2\leq i\leq n-1,\\[1mm]
[e_n,x]=-[x,e_n]=(n-4)e_n,&\\[1mm]
[e_i,y]=-[y,e_i]=e_i,& 2\leq i\leq n-1,\\[1mm]
[e_n,y]=-[y,e_n]=2e_n,\\[1mm]
\end{array}\right.$$

The algebra $R(g^{2}_{(n,1)}, 2)$ is obtained by applying similar arguments used for the algebra $R(g^{1}_{(n,1)}, 2)$.
\end{proof}

\subsection{Rigidity of algebras $R(g^{i}_{(n,1)}, 2), \ i=1, 2$}

\

In this part we prove the rigidity of the algebras $R(g^{1}_{(n,1)}, 2)$ and $R(g^{1}_{(n,1)}, 2)$.
First we prove the triviality of the second cohomology group with coefficient in itself for these algebras.

\begin{thm}
$H^2(R(g^{1}_{(n,1)}, 2),R(g^{1}_{(n,1)}, 2))=H^2(R(g^{2}_{(n,1)}, 2),R(g^{2}_{(n,1)}, 2))=\{0\}$.
\end{thm}

\begin{proof}
 Let $R$ be a solvable Lie algebra isomorphic to $R(g^{i}_{(n,1)}, 2), \ i=1, 2$. Then from the formula (\ref{eq1}), we have the following:
$$H^2(R,R)\simeq H^2(Q,\mathbb{C})\otimes H^0(N,R)^Q+H^1(Q,\mathbb{C})\otimes H^1(N,R)^Q+H^0(Q,\mathbb{C})\otimes H^2(N,R)^Q.$$

From this it is clear that if $H^i(N,R)^Q=\{0\}$ for all $0\leq i\leq 2$, then we get $H^2(R,R)=\{0\}$.
Since the center of $R$ is trivial, then $H^0(N,R)^Q=\{0\}$. Next, we   consider the cohomology groups $H^1(N,R)^Q$ and $H^2(NL,R)^Q$.

Let $f\in C^1(N,R)$, then $f$ has the from
$$f(e_i)=\sum\limits_{t=1}^{n}\alpha_{i,t}e_t+\beta_ix+\gamma_iy, \quad 1 \leq i \leq n, \quad e_i\in N.$$

Now we rewrite the formula (\ref{eq2}) for the elements $x\in Q$ follows

$$(d^1f)(x,e_i)=[x,f(e_i)]-f([x,e_i])=0.$$

From the equalities $(d^1f)(x,e_i)=(d^1f)(y,e_i)=0$ for $1\leq i\leq n$ we deduce
$$f(e_i)=\alpha_{i,i}e_i, \quad 1 \leq i \leq n, \quad e_i\in N.$$

this imply $\dim C^1(N,R)=n$. In order to compute the $Z^1(N,R)^Q$, we need the general form of a derivation.

From the equalities $(d^1f)(e_i,e_j)=0$ for $1\leq i,j\leq n$ we obtain
\begin{equation}\label{eq3}\left\{\begin{array}{ll}
f(e_1)=\alpha_{1,1}e_1,\\[1mm]
f(e_2)=\alpha_{2,2}e_2,\\[1mm]
f(e_i)=((i-2)\alpha_{1,1}+\alpha_{2,2})e_i,& 3\leq i \leq n-1,\\[1mm]
f(e_n)=(\varepsilon\alpha_{1,1}+2\alpha_{2,2})e_n,\\[1mm]
\end{array}\right.\end{equation}
where $\varepsilon=n-4$, if $R=R(g^{1}_{(n,1)}, 2)$ and $\varepsilon=0$, if $R=R(g^{2}_{(n,1)}, 2)$.

Further, from (\ref{eq3}) we deduce $\dim Z^1(N,R)^Q=2$.
It is easily seen that $\mathcal{R}_x$ and $\mathcal{R}_y$ are 1-coboundaries. As a consequence,
$$\dim H^1(N,R)^Q=0 \quad \text{and} \quad \dim B^2(N,R)^Q=\dim C^1(N,R)-\dim Z^1(N,R)^Q=n-2.$$

We now proceed to compute the 2-cocycles. For an arbitrary fixed cochain $\varphi\in C^2(N,R)$ we set
$$\varphi(e_i,e_j)=\sum^n_{k=1} a^k_{i,j}e_k+b_{i,j}x+c_{i,j}y, \quad 1\leq i,j\leq n.$$

Using formula (\ref{eq2}) for the elements $x\in Q$ we have
$$(d^2\varphi)(x,e_i,e_j)=[x,\varphi(e_i,e_j)]-\varphi([x,e_i],e_j)-\varphi(e_i,[x,e_j])=0.$$

in the case when the solvable algebra $R$ is isomorphic to the algebra $R(g^1_{(n,1)},2)$ by considering the equalities $(d^2\varphi)(x,e_i,e_j)=(d^2\varphi)(y,e_i,e_j)=0$,
 we obtain the following
$$\left\{\begin{array}{ll}
\varphi(e_1,e_i)=-\varphi(e_i,e_1)=a_{1,i}e_{i+1},&2\leq i\leq n-2,\\[1mm]
\varphi(e_i,e_{n-i})=-\varphi(e_{n-i},e_{i})=a_{i,n-i}e_{n},&2\leq i\leq n-2.\\[1mm]
\end{array}\right.$$

Further, the equalities $(d^2\varphi)(e_1,e_i,e_{n-i-1})=0$ imply the following restrictions:
$$a_{i+1,n-i-1}=(-1)^{i+1}(a_{1,n-i-1}-a_{1,i})-a_{i,n-i}, \quad  2\leq i \leq\frac{n-1}{2}.$$
In this case, the free parameters are $a_{1,2},a_{1,3},\dots,a_{1,n-2},a_{2,n-2}$.

For the case when the solvable algebras $R$ is isomorphic to the algebra $R(g^2_{(n,1)},2)$. Then, the considering the equalities
$(d^2\varphi)(x,e_i,e_j)=(d^2\varphi)(y,e_i,e_j)=(d^2\varphi)(e_1,e_i,e_{n})=0$,
 we obtain the following

$$\left\{\begin{array}{ll}
\varphi(e_1,e_i)=-\varphi(e_i,e_1)=a_{1,i}e_{i+1},&2\leq i\leq n-2,\\[1mm]
\varphi(e_i,e_{n})=-\varphi(e_{n},e_{i})=a_{i,n}e_{i+2},&2\leq i\leq n-3,\\[1mm]
\end{array}\right.$$
where,
$$a_{i+1,n}=a_{i,n}-a_{1,i}+a_{1,i+2}, \quad  2\leq i \leq n-4.$$
In this case, the free parameters are $a_{1,2},a_{1,3},\dots,a_{1,n-2},a_{2,n}$.

Therefore,
$$\dim H^2(N,R)^Q=\dim Z^2(N,R)^Q-\dim B^2(N,R)^Q=n-2-(n-2)=0.$$
\end{proof}

\begin{cor}
The algebras of Theorem \ref{thmmu2} are rigid.
\end{cor}
\begin{proof}
Let $R$ be a solvable Leibniz algebra of the list of Theorem \ref{thmmu2}. Since $\cent(R)=\{0\}$ and $H^2(R,R)=\{0\}$. From Theorem 2 in \cite{Fialowski} we conclude that $HL^2(R,R)=\{0\}$.  Due to Theorem \ref{Bal} this imply that $R$ is a rigid algebra.
\end{proof}

\begin{thm}
Let $R$ be a solvable Lie algebra isomorphic to one of the algebra from the list of Theorem \ref{thmmu3}. Then $\dim H^2(R,R)=1$.
\end{thm}
\begin{proof}
The following cocycles $\varphi$ are representatives of $H^2(R,R)$:

\begin{itemize}
\item for the solvable Lie algebra $R(g^3_{(n,1)},1)$:
$$\left\{\begin{array}{ll}
\varphi(e_1,e_n)=e_2,\\[1mm]
\varphi(e_i,e_n)=(i-3)e_{i+2},& 4 \leq i \leq n-3,\\[1mm]
\varphi(e_2,e_i)=1,5(i-4)e_{i+3},& 5 \leq i \leq n-4,\\[1mm]
\varphi(e_3,e_i)=-1,5e_{i+4},& 4 \leq i \leq n-5.\\[1mm]
\end{array}\right.$$

\item for the solvable Lie algebra $R(g^1_{7},1)$:
$$\varphi(e_1,e_6)=e_7,\ \varphi(e_3,e_4)=e_{7},$$

\item for the solvable Lie algebra $R(g^2_{9},1)$:
$$\varphi(e_1,e_8)=e_9,\ \varphi(e_2,e_6)=-24e_{8}, \ \varphi(e_3,e_4)=-6e_{7}, \ \varphi(e_3,e_5)=-6e_{8},$$
$$\varphi(e_3,e_6)=5e_9,\ \varphi(e_4,e_5)=-7e_{9},$$

\item for the solvable Lie algebra $R(g^3_{11},1)$:
$$\varphi(e_1,e_{10})=e_{11},\ \varphi(e_2,e_5)=-\frac{3}{4}e_{7}, \ \varphi(e_2,e_6)=-\frac{3}{2}e_{8}, \ \varphi(e_2,e_7)=-e_{9},$$
$$\varphi(e_2,e_8)=\frac{3}{4}e_{10},\ \varphi(e_3,e_4)=-\frac{3}{4}e_{7},\ \varphi(e_3,e_5)=-\frac{3}{4}e_{8},\ \varphi(e_3,e_6)=\frac{1}{2}e_{9},$$
$$\varphi(e_3,e_7)=\frac{7}{4}e_{10},\ \varphi(e_4,e_5)=-\frac{5}{4}e_{9},\ \varphi(e_4,e_6)=-\frac{5}{4}e_{10},\ \varphi(e_4,e_7)=-e_{11},$$
$$\varphi(e_5,e_6)=2e_{11}.$$
\end{itemize}
\end{proof}

Since $\cent(R)=\{0\}$ and $\dim H^2(R,R)=1$ for the algebras of the list of Theorem \ref{thmmu2}, then due to  \cite[Theorem 2]{Fialowski} we conclude $\dim HL^2(R,R)=1$.

\section*{Acknowledgements}
This work was supported by Agencia Estatal de Investigación (Spain), grant
MTM2016-79661-P (European FEDER support included, UE).

\end{document}